\documentclass{amsart}
\usepackage{amssymb}
\usepackage{graphicx}
\usepackage{enumerate}
\usepackage{multirow}
\usepackage{amsmath,color}
\usepackage{hyperref}
\usepackage{url}

\newtheorem{theorem}{Theorem}[section]

\newtheorem{proposition}[theorem]{Proposition}

\newtheorem{lemma}[theorem]{Lemma}

\newcommand{\G}{\Gamma}

\DeclareMathOperator{\tr}{tr}

\title{Laplacian spectral characterization of roses}

\author{Changxiang He}
\address{College of Science, University of Shanghai for Science
and Technology, China}
\email{changxiang-he@163.com}
\author{Edwin R. van Dam}
\address{Department of Econometrics and O.R., Tilburg University,
	 The Netherlands}
\email{Edwin.vanDam@uvt.nl}

\date{}
\begin{document}


\keywords{rose graphs, Laplacian spectrum, closed walks, Sachs' theorem, matchings \\ This version will appear in Linear Algebra and its Applications, \url{https://doi.org/10.1016/j.laa.2017.08.012}}

\begin{abstract}
A rose graph is a graph consisting of cycles that all meet in one vertex. We show that except for two specific examples, these rose graphs are determined by the Laplacian spectrum, thus proving a conjecture posed by Lui and Huang [F.J. Liu and Q.X. Huang, Laplacian spectral characterization of $3$-rose graphs, {\sl Linear Algebra Appl.} 439 (2013), 2914--2920]. We also show that if two rose graphs have a so-called universal Laplacian matrix with the same spectrum, then they must be isomorphic. In memory of Horst Sachs (1927-2016), we show the specific case of the latter result for the adjacency matrix by using Sachs' theorem and a new result on the number of matchings in the disjoint union of paths.
\end{abstract}

\maketitle

\section{Introduction}

For $k \geq 2$, a $k$-rose graph is a graph with $k$ cycles that all meet in one vertex, which we call the central vertex. Alternatively, it is a connected graph with one vertex of degree $2k$, the central vertex, and for which all other vertices have degree $2$. We will show that except for two specific examples, these rose graphs are determined by the Laplacian spectrum. This proves a conjecture posed by Liu and Huang \cite{F. J. Liu}, who also showed that all $3$-rose graphs are indeed determined by the Laplacian spectrum. Also rose graphs in which each cycle is a triangle are determined by the Laplacian spectrum, by Liu, Zhang, and Gui \cite{X. Liu}; such graphs are better known as friendship graphs. The result for triangle-free $2$-rose graphs (also known as $\infty$-graphs) was proven by Wang, Huang, Belardo, and Li Marzi \cite{J. F. Wang}. Not all $2$-rose graphs are however determined by the Laplacian spectrum. One counterexample is given in Figure \ref{fig:cosp34}; it already appeared in the thesis of one of the authors \cite{EvDthesis} as an example of a pair of cospectral graphs of which one is bipartite and the other is not. Another (and final) counterexample is given in Figure \ref{fig:cosp35}.

\begin{figure}[!htb]
\begin{center}
\includegraphics[width=12cm]{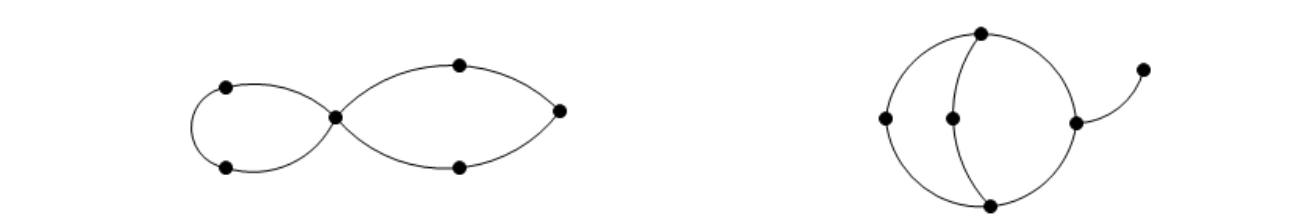}
\makeatletter\def\@captype{figure}\makeatother \caption{The rose graph $R(3,4)$ and the other graph with Laplacian spectrum $\{0,\ 3-\sqrt{5},\ 2,\ 3,\ 3,\ 3+\sqrt{5}\}$.}\label{fig:cosp34}
\end{center}
\end{figure}
\begin{figure}[!htb]
\begin{center}
\includegraphics[width=12cm]{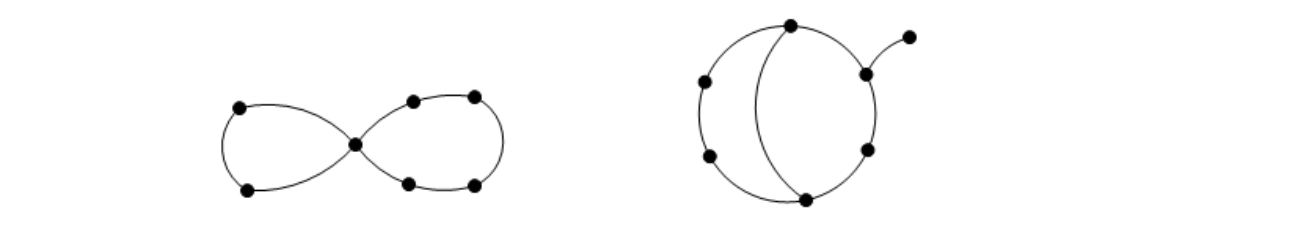}
\makeatletter\def\@captype{figure}\makeatother \caption{The rose graph $R(3,5)$ and the other graph with Laplacian spectrum $\{0,0.6086...,\frac{\ 5-\sqrt{5}}{2},2.227...,3,\frac{\ 5+\sqrt{5}}{2},5.164...\}$. }\label{fig:cosp35}
\end{center}
\end{figure}

The problem of which rose graphs are determined by the Laplacian spectrum is part of a more general problem of which graphs are determined by the spectrum. For an overview and motivation of this problem, we refer the reader to the survey papers by Van Dam and Haemers \cite{DH03,DH09}.
The analogous problem of determining which $k$-rose graphs are determined by the {\em signless Laplacian} spectrum has recently been solved by Liu, Shan, and Das (private communication; for the particular cases $k=2,3,$ and $4$, see \cite{J. F. Wang}, \cite{wang3}, and \cite{MaHuang}, respectively).

This paper is built up as follows. In Section \ref{sec:pre}, we will give some preliminary results that we will use in the final two sections. We note that except for Lemma \ref{laplacedeterminesnumbers} and Sachs' theorem, which are both well-known, our paper is mostly self-contained. In Section \ref{sec:universal}, we will introduce a universal Laplacian matrix; this is a class of matrices that includes well-studied matrices such as the adjacency matrix, the Laplacian matrix, and the signless Laplacian matrix. We show in Proposition \ref{genLap} that two rose graphs with the same universal Laplacian spectrum must be isomorphic. This is an important step towards the main result. In Section \ref{sec:Sachs}, which is in memory of Horst Sachs (1927-2016), we sketch an alternative proof of the result in Proposition \ref{genLap} for the adjacency matrix by using Sachs' theorem \cite{Sachs} and a lemma on the number of matchings in a disjoint union of paths. In Section \ref{sec:case3+}, we derive the Laplacian spectral characterization of $k$-rose graphs for $k \geq 3$. The case $k=2$, which is more complicated, will be discussed in the final section. To summarize, we prove the following, where
$R(3,\ell)$ denotes the $2$-rose graph with one triangle and one cycle of length $\ell$, with $\ell \geq 3$.

\begin{theorem}\label{thm:general} Let $k \geq 2$. All $k$-rose graphs, except for $R(3,4)$ and $R(3,5)$, are determined by the Laplacian spectrum. Both $R(3,4)$ and $R(3,5)$ have one Laplacian cospectral mate.
\end{theorem}

\section{Preliminaries}\label{sec:pre}

For a graph $\G$, we let $A$ be its adjacency matrix, $D$ the diagonal matrix of vertex degrees, and $L=D-A$ its Laplacian matrix.

In order to derive our results, we will use some preliminary results. The first is very elementary and well known; see \cite[Lemma 14.4.3]{BrHa}, for example.

\begin{lemma}\label{laplacedeterminesnumbers} The Laplacian spectrum of a graph determines its number of vertices, number of edges, number of components, and number of spanning trees.
\end{lemma}

Also the following is known about the vertex degrees and the number of triangles in the graph; see \cite{F. J. Liu}. Because the proof given by Liu and Huang \cite{F. J. Liu} is quite technical, we provide a new, elementary proof of the second item. A similar result and proof was given for the signless Laplacian spectrum by Cvetkovi\'c, Rowlinson, and Simi\'c \cite[Cor.~4.3]{CRS} and for the Laplacian spectrum of signed graphs by Belardo and Petecki \cite[Thm.~3.4]{BP}.

\begin{lemma}\label{laplacetriangles} Let $\G$ be a graph with $n$ vertices, $t$ triangles, and vertex degrees $d_i$ for $i=1,2,\dots,n$. Then the following numbers are determined by the Laplacian spectrum of $\G$.
\begin{enumerate}[(i)]
\item $\sum_{i=1}^n d_i^2$,
\item $\sum_{i=1}^n d_i^3 - 6t$.
\end{enumerate}
\end{lemma}

\begin{proof} Let $L=D-A$ be the Laplacian matrix of $\G$. Then the spectrum of $L$ determines $\tr L^j$ for every positive integer $j$.

By working out $L^2$, we obtain that
$$\tr L^2=\tr D^2- 2\tr DA+ \tr A^2=\sum_{i=1}^n d_i^2 +\sum_{i=1}^n d_i.$$
Because $\sum_{i=1}^n d_i =\tr L$ (which equals twice the number of edges), and this is determined by the Laplacian spectrum, it follows that also $\sum_{i=1}^n d_i^2$ is determined by the Laplacian spectrum.

Similarly, by working out $L^3$, we obtain that
$$\tr L^3=\tr D^3-3 \tr D^2A +3\tr DA^2-\tr A^3=\sum_{i=1}^n d_i^3+3 \sum_{i=1}^n d_i^2 -6t,$$
which proves ($ii$) (by using ($i$)).
\end{proof}

The following result seems new; we will apply this in Section \ref{sec:case k=2}.

\begin{lemma}\label{laplacequadrangles} Let $\G$ be a graph with $n$ vertices, $t$ triangles, $f$ $4$-cycles, and vertex degrees $d_i$ for $i=1,2,\dots,n$. Furthermore, let $t_i$ be the number of triangles through $i$, for $i=1,2,\dots,n$. Then
$$\sum_{i=1}^n d_i^4 + 2\sum_{i=1}^n\sum_{j \sim i}d_id_j - 8\sum_{i=1}^n d_it_i +24t+8f$$ is determined by the Laplacian spectrum of $\G$.
\end{lemma}

\begin{proof} Similar as in the proof of Lemma \ref{laplacetriangles}, this follows from working out $\tr L^4$. Moreover, we will use that $(A^3)_{ii}=2t_i$ and that $\tr A^4=2\sum_{i=1}^n d_i^2-\sum_{i=1}^n d_i+8f$. The latter follows from counting the number of closed walks of length 4. Then
\begin{align*}
\tr L^4 & = \tr D^4 -4\tr D^3A +4\tr D^2A^2 +2\tr DADA -4\tr DA^3 +\tr A^4\\
&= \sum_{i=1}^n d_i^4 + 4\sum_{i=1}^n d_i^3 + 2\sum_{i=1}^n\sum_{j \sim i}d_id_j - 8\sum_{i=1}^n d_it_i +2 \sum d_i^2-\sum d_i+8f.
\end{align*}
Now the result follows by applying Lemma \ref{laplacetriangles} (and again that $\sum_{i=1}^n d_i$ equals twice the number of edges, which is also determined by the Laplacian spectrum).
\end{proof}

\section{A universal Laplacian matrix}\label{sec:universal}

For a graph $\G$ with adjacency matrix $A$ and degree matrix $D$, and fixed $\alpha$ and $\beta \neq 0$, we let $Q=Q(\alpha,\beta)=\alpha D+\beta A$. We call the matrix $Q$ a universal Laplacian matrix. Particular cases are the adjacency matrix $Q(0,1)$, the Laplacian matrix $Q(1,-1)$, and the signless Laplacian matrix $Q(1,1)$.
We note that Wang, Li, Lu, and Xu \cite{WeiWang} showed (among other results) that rose graphs are determined by the {\em set of spectra} of all universal Laplacian matrices, or in their terminology, by the generalized characteristic polynomial. Here we will show a result for each separate spectrum.

\begin{proposition}\label{genLap} Fix $\alpha$ and $\beta \neq 0$. Let $\G$ and $\G'$ be rose graphs that are cospectral with respect to the universal Laplacian matrix $Q(\alpha,\beta)$. Then $\G$ and $\G'$ are isomorphic.
\end{proposition}

\begin{proof} Without loss of generality, we take $\beta=1$. Let $Q=A+\alpha D$ and $Q'=A'+\alpha D'$ be the universal Laplacian matrices of $\G$ and $\G'$, respectively. Note first that the number of edges of $\G$ and $\G'$ must be the same. Indeed, for $\alpha \neq 0$ this follows from the fact that $\tr Q=\tr Q'$; for $\alpha = 0$ it follows from $\tr A^2=\tr A'^2$. Therefore, both $\G$ and $\G'$ are $k$-rose graphs for a certain $k$, and thus they have the same degree sequence.

We will then use that $\tr Q^{\ell}=\tr Q'^{\ell}$ for all $\ell$, and use induction to show that $\G$ and $\G'$ have the same number of $\ell$-cycles. Because both graphs are $k$-rose graphs, this implies that they are isomorphic.

To prove both the basis $\ell=3$ and the induction step, let $\ell \geq 3$, and assume that $\G$ and $\G'$ have equally many $j$-cycles for each $j<\ell$. Note that this is a valid assumption for the basis $\ell=3$.

 By working out $Q^{\ell}$ in terms of $A$ and $D$, we obtain that $$\tr A^{\ell}+ \tr B =\tr A'^{\ell}+\tr B'$$ for certain matrices $B$ and $B'$ that are weighted sums of products $B_t$ of the matrices $D$ and $A$, and $B'_t$ of $D'$ and $A'$, respectively, with $t\in T$ for some index set $T$. The particular weights in $B$ and $B'$ are the same; moreover in $B_t$ and $B'_t$, the matrices $A$ and $A'$ appear less than $\ell$ times. Because the product $B_t$ also contains the degree matrix $D$, the trace of $B_t$ counts a weighted number of closed walks of given length less than $\ell$ with weights from the degree matrix. For example, $\tr AD^2A^3D$ counts the weighted number of closed walks $u_0 \sim u_1 \sim u_2 \sim u_3 \sim u_4=u_0$ of length $4$, weighted with $d_1^2d_4$ (where $d_i$ is the degree of $u_i$ for $i=1$ and $4$).

Suppose now that $\G$ and $\G'$ have a different number of $\ell$-cycles. Without loss of generality, we assume that $\G$ has more $\ell$-cycles than $\G'$. We will show that this implies that
\begin{equation}\label{ellwalks}
\tr A^{\ell}>\tr A'^{\ell},
\end{equation}
whereas
\begin{equation}\label{smallwalks}
\tr B = \tr B',
\end{equation}
which gives a clear contradiction (and thus proves that $\G$ and $\G'$ have the same number of $\ell$-cycles). In order to show this, we will give a bijection $\phi: V(\G) \rightarrow V(\G')$ that preserves the degree and the weighted number of closed walks of given length $j$ less than $\ell$ that start at a given vertex, thus showing \eqref{smallwalks}.
Moreover, for many vertices the number of closed walks of length $\ell$ that start at a given vertex will be preserved, but for some the number will decrease, thus showing \eqref{ellwalks}.

\begin{figure}[!htp]
\begin{center}
\includegraphics[width=12cm]{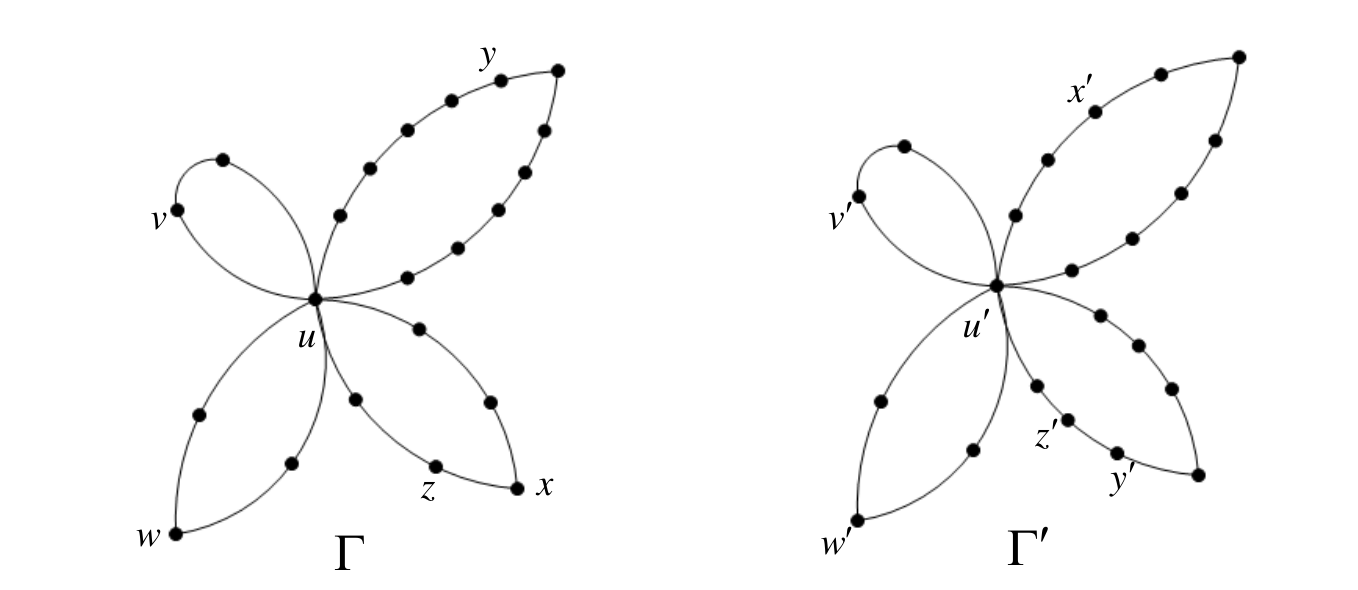}
\end{center}
\caption{The $4$-rose graphs $\G$ and $\G'$; $\ell=6$}\label{fig:twok}
\end{figure}

The map $\phi$ is built up as follows; see also Figure \ref{fig:twok} for an example with $\ell=6$.
We consider the subgraph of $\G'$ induced by the vertices on the $j$-cycles, with $j \leq \ell$. In other words, we remove the larger cycles, of which there are $s$, say. It is clear that $\G$ has a subgraph isomorphic to this subgraph of $\G'$. We now use a fixed isomorphism in the canonical way for the map $\phi$ (on the vertices of the subgraph of $\G$). In case there is just one cycle, we make sure that the central vertex of $\G$ (i.e., the unique one with degree $2k$) is mapped to the central vertex of $\G'$ (in the other cases, this goes automatically). For the remaining vertices, which are all on cycles of length at least $\ell$, we make sure that $\phi$ preserves the distance to the central vertex, as long as this distance is less than $\ell/2$. Because both $\G$ and $\G'$ have the same number $s$ of ``remaining (larger) cycles'', this is possible. Vertices of $\G$ that are at distance precisely $\ell/2$ from the central vertex are also mapped to vertices of $\G'$ that are at distance precisely $\ell/2$ from the central vertex (note that $\G'$ has more such vertices than $\G$; if $\ell$ is even). The remaining vertices of $\G$ are mapped arbitrarily, but one-to-one, to the remaining vertices of $\G'$.

Let us now argue that $\phi$ has the required properties to show \eqref{ellwalks} and \eqref{smallwalks}. First of all, it is clear that $\phi$ preserves the degrees of vertices. Now let us consider the weighted number $w_j(u)$ of closed walks of length $j$ that starts at a vertex $u$ of $\G$, for $j \leq \ell$. For such a vertex $u$, let us abbreviate $\phi(u)$ by $u'$. Note now that in every closed walk of length $j$ that starts in a given vertex $u$ (in any graph), there are only vertices involved that are at distance at most $j/2$ from $u$, and no edges are involved between those vertices that at are distance $j/2$ from $u$. We therefore define the {\em weighted rooted closed walk graph} $C_j(u)$ as the subgraph of $\G$ induced on the vertices at distance at most $j/2$ from $u$ from which the edges, if any, between the vertices at distance exactly $j/2$ from $u$ have been removed. Moreover, we assign $u$ as its root and take the degree of a vertex in the original graph as its weight. With this definition, $w_j(u)$ equals the weighted number of closed walks of length $j$ in $C_j(u)$ that start in its root. Similarly, we define $C'_j(u')$ as a weighted rooted closed walk graph of $\G'$. In order to compare $w_j(u)$ to $w'_j(u')$, we can now restrict to comparing weighted closed walks in $C_j(u)$ and $C'_j(u')$.

For $j<\ell$, it is straightforward to check that the weighted rooted graphs $C_j(u)$ and $C'_j(u')$ are isomorphic for all $u$ (where we recall that $u'=\phi(u)$), and hence $w_j(u)=w'_j(u')$ for all $u$. Thus, \eqref{smallwalks} follows.

The situation for $j=\ell$ is more subtle of course, but except for the central vertex and those vertices $u$ that are on an $\ell$-cycle while $u'$ is not, the {\em unweighted} rooted closed walk graphs $C_j(u)$ and $C'_j(u')$ are again isomorphic. Note that in this case we count unweighted closed walks, and indeed, this implies that $A^{\ell}_{uu}=A'^{\ell}_{u'u'}$, except in the described cases. If $u$ is on an $\ell$-cycle while $u'$ is not, then it in fact follows that $A^{\ell}_{uu}=A'^{\ell}_{u'u'}+2$, because there are two extra closed walks; those walking around the $\ell$-cycle (with two directions possible). Also if $u$ is the central vertex, then $A^{\ell}_{uu}>A'^{\ell}_{u'u'}$. All together, this shows \eqref{ellwalks}. As argued before, this gives a contradiction, thus proving the induction, and therefore that $\G$ and $\G'$ are isomorphic.
\end{proof}

\section{In memory of Horst Sachs (1927-2016)}\label{sec:Sachs}

For the adjacency matrix, the result in Proposition \ref{genLap} can also be obtained in a different way, namely by using Sachs' theorem \cite{Sachs}. In memory of professor Horst Sachs, one of the authors of the influential monograph \cite{CDS} on spectra of graphs, we include a sketch of this alternative proof here. Let us first recall Sachs' theorem.

Given a graph $\G$ on $n$ vertices and a positive integer $i$, a Sachs $i$-subgraph $S$ of $\G$ is a disjoint union of cycles and edges on a total of $i$ vertices. The number of components of $S$ and the number of cycles in $S$ are denoted by $k(S)$ and $c(S)$, respectively. We denote the set of Sachs $i$-subgraphs of $\G$ by $\mathcal{S}_i(\G)$. Sachs' theorem expresses the coefficient $a_i$ of (the term $x^{n-i}$ of) the characteristic polynomial of $\G$ in terms of its Sachs subgraphs as
\begin{equation}\label{sachscoef}
a_i=\sum_{S \in \mathcal{S}_i(\G)}(-1)^{k(S)}2^{c(S)}.
\end{equation}
For example, from this it follows that $a_1=0$, $-a_2$ equals the number of edges of $\G$, and $-a_3$ equals twice the number of triangles of $\G$.

In the rose graphs that we are considering in this paper, a Sachs subgraph is either the disjoint union of one cycle and a matching (a disjoint union of edges), or simply a matching. For both cases, counting Sachs subgraphs boils down to counting matchings in disjoint unions of paths. Indeed, once a cycle is chosen as subgraph (and ``removed" from the rose graph), a disjoint union of paths remains in which matchings have to be counted. To count the number of $j$-matchings (disjoint unions of $j$ edges) in a $k$-rose graph, we can distinguish between those matchings that contain an edge through the central vertex and those that do not. To be more precise, let $m(\G,j)$ denote the number of $j$-matchings in a graph $\G$. If $\G$ is a $k$-rose graph with central vertex $c$ that is adjacent to vertices $u_i$ for $i=1,2,\dots, 2k$, then it is clear that
$$m(\G,j)=m(\G-c,j)+\sum_{i=1}^{2k}m(\G-c-u_i,j-1),$$
where $\G-c$ is the graph obtained from $\G$ by removing $c$ (and incident edges), and similarly $\G-c-u_i$ is the graph obtained by removing $c$ and $u_i$, for $i=1,2,\dots, 2k$. Clearly all these graphs are disjoint unions of paths. For our result, we do not have to derive the actual number of matchings in a disjoint union of paths (or in a $k$-rose graph); instead it will suffice to use the following.

\begin{lemma}\label{samenumber} Let $n, k, i$ be positive integers. The number of $i$-matchings in any disjoint union of $k$ paths, each on at least $i$ vertices, on a total of $n$ vertices, is independent of the lengths of the individual paths.
\end{lemma}
\begin{proof} In order to prove this, we use induction on $i$. For $i=1$, we have to count the number of edges, which is clearly the same, i.e., $n-k$, for every disjoint union of $k$ paths on a total of $n$ vertices.

To prove the induction step, let $i>1$, and assume that the result is true for $i-1$. let $\G=P_{\ell_1}\sqcup P_{\ell_2}\sqcup\cdots\sqcup P_{\ell_k}$  and  $\G'=P_{\ell'_1}\sqcup P_{\ell'_2}\sqcup\cdots\sqcup P_{\ell'_k}$ be two graphs that are disjoint unions of $k$ paths, with $\ell_j \geq i$ and $\ell'_j \geq i$ for every $j=1,2,\dots,k$ and $\sum_{j=1}^k \ell_j=\sum_{j=1}^k \ell'_j=n$.

Let $\G''=P_{\ell''_1}\sqcup P_{\ell''_2}\sqcup\cdots\sqcup P_{\ell''_k}$, where $\ell''_j=\min\{\ell_j,\ \ell'_j\}$ for $j=1,2,\dots,k$.
The graph $\G''$ can be obtained from both $\G$ and $\G'$, by deleting the same number of vertices, $s$ say. In fact, this can be done by a sequence of $s$ steps, where in each step $h=1,2,\dots,s$ we remove a pendant vertex $u_h$ and its corresponding edge $e_h=u_hv_h$ from $\G$ and pendant vertex $u'_h$ and its corresponding edge $e'_h=u'_hv'_h$ from $\G'$. Thus, we have sequences of graphs $(\G_h)_{h=0}^s$ and $(\G'_h)_{h=0}^s$, where $\G_0=\G$, $\G'_0=\G'$, $\G_h=\G_{h-1}-u_h$, $\G'_h=\G'_{h-1}-u'_h$, for $h=1,2,\dots,s$, and finally $\G_s=\G'_s=\G''$. In this way, $\G_h$ and $\G'_h$ are disjoint unions of $k$ paths, and each path is on at least $i$ vertices, for every $h=0,1,\dots,s$.

By distinguishing between matchings that do not contain any of the edges $e_h$, for $h=1,2,\dots,s$, and those that do, we obtain the expression
\begin{equation}\label{matchingi}
m(\G,i)=m(\G'',i)+\sum_{h=1}^s m(\G_{h}-v_h,i-1).
\end{equation}
Indeed, $m(\G'',i)$ is the number of $i$-matchings of $\G$ that do not contain any of the edges $e_h$, for $h=1,2,\dots,s$, whereas $m(\G_{h}-v_h,i-1)$ is the number of $i$-matchings of $\G$ that contain $e_h$, but none of the edges $e_1,e_2,\dots,e_{h-1}$. Similarly, we obtain that
\begin{equation}\label{matchingi'}
m(\G',i)=m(\G'',i)+\sum_{h=1}^s m(\G'_{h}-v'_h,i-1).
\end{equation}

By the induction hypothesis, and because $\G_h-v_h$ and $\G'_h-v'_h$ have the same number of vertices and they are both disjoint unions of paths, each on at least $i-1$ vertices, it now follows from \eqref{matchingi} and \eqref{matchingi'} that $m(\G,i)=m(\G',i)$, which finishes the proof.
\end{proof}

We note that the condition that each of the paths has at least $i$ vertices is necessary in general. For example, $P_1\sqcup P_3$ has no $2$-matchings, whereas $P_2\sqcup P_2$ does have one. Still, we will use the above result in cases where the two disjoint unions of paths may have some paths of too small a size. However, the two graphs will only differ in the ``larger'' paths, and in that case using the above result still gives the same numbers of matchings. Indeed, it follows easily from the lemma that if $\G$ and $\G'$ are as in the above proof, and $\G'''$ is some other graph (in particular, a disjoint union of ``small'' paths), then the numbers of $i$-matchings in $\G \sqcup \G'''$ and in $\G'\sqcup \G'''$ are the same.

Let us now finish the sketch of the proof. Consider two $k$-rose graphs $\G$ and $\G'$ that have the same adjacency spectrum, and hence the same coefficients $a_{\ell}$ and $a'_{\ell}$, for $\ell=1,2,\dots,n$ of the characteristic polynomial. Like in the proof of Proposition \ref{genLap}, we aim to show by induction on $\ell$ that the two graphs have the same number of $\ell$-cycles.
Let $\ell \geq 3$ and assume that the two graphs have the same number of $j$-cycles for $j < \ell$. Let us consider Sachs $\ell$-subgraphs of $\G$ and $\G'$, respectively. Let $3 \leq j<\ell$. Because the number of $j$-cycles in $\G$ and $\G'$ is the same, it follows from Lemma \ref{samenumber} and the earlier arguments that $\G$ and $\G'$ have the same number of Sachs subgraphs of the form $C_j\sqcup \frac{\ell-j}2K_2$. Indeed, as indicated, once $C_j$ is removed, what remains is a disjoint union of paths. There may be paths of size at most $\ell-2$ coming from cycles of size at most $\ell-1$, but there are equally many of each length in the two graphs under consideration, and hence the number of $\frac{\ell-j}2$-matchings is the same. The same argument applies to show that the two graphs have the same number of $\frac{\ell}2$-matchings (when $\ell$ is even). Because all these numbers of the above types of Sachs $\ell$-subgraphs are the same for $\G$ and $\G'$, it follows from \eqref{sachscoef} that the same must be true for the remaining type of Sachs $\ell$-subgraphs: the $\ell$-cycle. Thus, $\G$ and $\G'$ have the same number of $\ell$-cycles, which finishes the sketch of the proof.

\section{The characterization by the Laplacian spectrum}\label{sec:case3+}

In this section, we will prove our main result, that is, that every $k$-rose graph with $k \geq 3$ is determined by the Laplacian spectrum. In order to do this, we need two lemmas about the degree sequence of a graph that is cospectral to a $k$-rose graph, after which we can apply Proposition \ref{genLap}.

\begin{lemma}\label{maxdegree} Let $k \geq 2$ and let $\G$ be a graph with the same Laplacian spectrum as a $k$-rose graph. If $\G$ has maximum degree at least $2k$, then $\G$ is a $k$-rose graph.
\end{lemma}
\begin{proof} Denote the vertex degrees of $\G$ by $d_i$, for $i=1,2,\dots,n$. Because $\G$ is cospectral to a $k$-rose graph, which has vertex degrees $2k$ (once) and 2 ($n-1$ times), it follows from applying Lemmas \ref{laplacedeterminesnumbers} and \ref{laplacetriangles} that
\begin{equation}\label{sumsdegrees} \sum_{i=1}^n (d_i-2)^2= (2k-2)^2. \end{equation}
If $\G$ has maximum degree at least $2k$, then it follows from \eqref{sumsdegrees} that one of the degrees equals $2k$, and the other degrees equal $2$. Because $\G$ is connected by Lemma \ref{laplacedeterminesnumbers}, $\G$ must therefore be a $k$-rose graph.
\end{proof}

\begin{lemma}\label{inequality}
Let $k\geq 3$, let $\G$ be a graph with the same Laplacian spectrum as a $k$-rose graph, and let the vertex degrees of $\G$ be denoted by $d_i$, for $i=1,2,\dots,n$. If $\G$ is not a rose graph, then
 $$\sum_{i=1}^n (d_i-2)^3< (2k-2)^3-6k.$$
 \end{lemma}
 \begin{proof}By Lemma \ref{maxdegree}, we may assume that $d_i \leq 2k-1$ for $i=1,2,\dots,n$. Using \eqref{sumsdegrees}, it follows that $\sum_{i=1}^n(d_i-2)^3 \leq (2k-3)\sum_{i=1}^n(d_i-2)^2=(2k-3)(2k-2)^2$. For $k>3$, the required statement now follows easily. For $k=3$, the argument is a bit more technical. In this case $\sum_{i=1}^n(d_i-2)^2=16$, which implies that $d_i-2=3=2k-3$ for at most one value of $i$. Now it follows that $\sum_{i=1}^n(d_i-2)^3 \leq 3^3 +2 (\sum_{i=1}^n(d_i-2)^2 -3^2)=41<46= (2k-2)^3-6k$.
 \end{proof}

\begin{theorem}\label{cospctral} For $k\geq 3$, every $k$-rose graph is determined by its Laplacian spectrum.
\end{theorem}
\begin{proof}Fix $k$ and suppose, on the contrary, that $\G$ is a graph with the same Laplacian spectrum as a given $k$-rose graph $\G'$, but that $\G$ and $\G'$ are not isomorphic. By Proposition \ref{genLap}, we may assume that $\G$ is not a rose graph.
As before, we denote the vertex degrees of $\G$ by $d_i$ for $i=1,2,\dots,n$. Furthermore, we let $t$ and $t'$ be the number of triangles in $\G$ and $\G'$, respectively. From Lemmas \ref{laplacedeterminesnumbers} and \ref{laplacetriangles}, it follows that
$\sum_{i=1}^n (d_i-2)^3 -6t=(2k-2)^3-6t'$. However, from Lemma \ref{inequality}, it follows that
 $$\sum_{i=1}^n (d_i-2)^3-6t< (2k-2)^3-6(k+t).$$ Therefore $t'>t+k$, but this is clearly impossible, because the number of triangles $t'$ in a $k$-rose graph can be at most $k$. Thus, we have a contradiction, which finishes the proof.
\end{proof}

\section{The lemniscate graphs or $2$-rose graphs}\label{sec:case k=2}

What remains are the $2$-rose graphs, also known as $\infty$-graphs (lemniscate graphs). Wang, Huang, Belardo, and Li Marzi \cite{J. F. Wang} already showed that the triangle-free $2$-rose graphs are determined by the Laplacian spectrum. Here we will show that all $2$-rose graphs, except for $R(3,4)$ and $R(3,5)$, are determined by the Laplacian spectrum, and that for both of these exceptions there is one Laplacian cospectral mate. Together with Theorem \ref{cospctral} this shows Theorem \ref{thm:general}.

\begin{theorem}\label{thm:2rose} All $2$-rose graphs, except for $R(3,4)$ and $R(3,5)$, are determined by the Laplacian spectrum. Both $R(3,4)$ and $R(3,5)$ have one Laplacian cospectral mate.
\end{theorem}
\begin{proof}
Let $\G'$ be a $2$-rose graph on $n$ vertices, and let $\G$ be a graph with the same Laplacian spectrum as $\G'$, but that is not isomorphic to $\G'$. From Proposition \ref{genLap}, it follows that that $\G$ is not a $2$-rose graph. It is clear from Lemma \ref{laplacedeterminesnumbers} that $\G$ is connected with $n$ vertices and $n+1$ edges, so it is bicyclic. As before, we denote the vertex degrees of $\G$ by $d_i$ for $i=1,2,\dots,n$, and we let $t$ and $t'$ be the number of triangles in $\G$ and $\G'$. It follows that $\sum_{i=1}^n (d_i-2)=2$, which together with \eqref{sumsdegrees} gives that the degree sequence of $\G$ must be  $3$ ($3$ times), $2$ ($n-4$ times), and $1$ (once). From Lemma \ref{laplacetriangles}, it now follows that $t=t'-1$.
Hence, if $\G'$ is triangle-free, then $t=-1$, which is clearly impossible. Indeed, this is (roughly speaking) the same argument as used by Wang, Huang, Belardo, and Li Marzi \cite{J. F. Wang} to prove that the triangle-free $2$-rose graphs are determined by the Laplacian spectrum.

If $\G'=R(3,3)$, then $t=1$. But there is no graph with five vertices, six edges, and one triangle. Thus, also $R(3,3)$ is determined by the Laplacian spectrum.

What remains are the graphs $\G'=R(3,n-2)$, with $n>5$. Such graphs have one triangle, so $\G$ is triangle-free. Let $f$ be the number of $4$-cycles in $\G$. We now intend to show that $f \geq 2$.
A straightforward calculation and applying Lemma \ref{laplacequadrangles} gives that
\begin{equation}\label{didj1}
180+16n + 2\sum_{i=1}^n\sum_{j \sim i}d_id_j +8f = 280+32n+8\delta_{n,6},
\end{equation}
where $\delta_{n,6}$ is the Kronecker delta indicating whether $n=6$ or not (that is, whether there is a $4$-cycle in $R(3,n-2)$ or not).
Now let $e_{33}$ be the number of edges between vertices of degree $3$ and let $e_{13}$ be the number of neighbors with degree $3$ of the vertex of degree $1$.
Then it follows that
\begin{equation}\label{didj2}
\sum_{i=1}^n\sum_{j \sim i}d_id_j =40+8n+2(e_{33}-e_{13}).
\end{equation}
Indeed, this can be derived in the following way. Let $V_h$ be the set of vertices of degree $h$, for $h=1,2,3$. Partition the adjacency matrix $A$ accordingly into blocks $A_{h\ell}$ for $h,\ell=1,2,3$, and let $a_{h\ell}$ be the sum of all entries of $A_{h\ell}$. It is clear that $\sum_{i=1}^n\sum_{j \sim i}d_id_j=\sum_{h,\ell}h\ell a_{h\ell}$ and $a_{h\ell}=a_{\ell h}$ for all $h,\ell=1,2,3$. Moreover, the following equations hold: $a_{11}+a_{12}+a_{13}=1$, $a_{21}+a_{22}+a_{23}=2(n-4)$, and $a_{31}+a_{32}+a_{33}=9$. Finally, note that $a_{11}=0,a_{13}=e_{13}$, and $a_{33}=2e_{33}$. By using all of this, we obtain \eqref{didj2}.

Putting \eqref{didj1} and \eqref{didj2} together, we now find that
\begin{equation}\label{eef}
e_{33}-e_{13}+2f=5+2\delta_{n,6}.
\end{equation}
We also know that $e_{33} \leq 2$ because $\G$ is triangle-free, so \eqref{eef} implies that $f \geq 2$. But $\G$ is bicyclic, so we also have that $f \leq 3$.

If $f=3$, then $\G$ has the complete bipartite graph $K_{2,3}$ as an induced subgraph. It implies that the number of spanning trees of $\G$ equals $12$. On the other hand, the number of spanning trees of $\G'$ equals $3(n-2)$, so $n=6$ by Lemma \ref{laplacedeterminesnumbers}. Given that $\G$ has $K_{2,3}$ as an induced subgraph, the vertex of degree $1$ must be adjacent to a vertex of degree $3$ and we obtain the graph on the right in Figure \ref{fig:cosp34}, which is indeed cospectral to $R(3,4)$.

Finally, we let $f=2$. If the two $4$-cycles do not share an edge, then the number of spanning trees of $\G$ equals $16$, but this cannot equal $3(n-2)$, the number of spanning trees of $R(3,n-2)$. So the two $4$-cycles share one edge (note that if they would share two edges, then $f=3$). Now the number of spanning trees of $\G$ equals $15$, which implies that $n=7$. Clearly, we now get the graph on the right in Figure \ref{fig:cosp35}, which is indeed cospectral to $R(3,5)$.
\end{proof}
\section*{Acknowledgement}

 This
work was completed while the first author was visiting the Department of Econometrics and Operations Research of  Tilburg University, for which support from the China Scholarship Council is gratefully acknowledged. The first author would like to thank Shanghai key laboratory of contemporary optics system for the support received during 2016.

\end{document}